\documentclass{amsart}
\usepackage{amsfonts,amssymb,amsmath}

\newtheorem{theorem}{Theorem}[section]
\newtheorem{lemma}[theorem]{Lemma}
\newtheorem{proposition}[theorem]{Proposition}
\newtheorem{corollary}[theorem]{Corollary}
\theoremstyle{definition}
\newtheorem{definition}[theorem]{Definition}
\newtheorem{example}[theorem]{Example}
\newtheorem{remark}[theorem]{Remark}

\begin{document}

\title[Separating maps on $\mathrm{AC(X,E)}$]{\bf Separating maps between spaces of vector-valued
absolutely continuous functions}

\author{Luis Dubarbie}

\address{Departamento de Matem\'aticas, Estad\'istica y
Computaci\'on, Facultad de Ciencias, Universidad de Cantabria,
Avenida de los Castros s/n, E-39071, Santander, Spain.}

\email{dubarbiel@unican.es}

\thanks{Research supported by the Spanish Ministry
of Science and Education (MTM2006-14786) and by a predoctoral grant from
the University of Cantabria and the Goverment of Cantabria.}



\begin{abstract}
In this paper we give a description of separating or disjointness
preserving linear bijections on spaces of vector-valued absolutely
continuous functions defined on compact subsets of the real line. We
obtain that they are continuous and biseparating in the
finite-dimensional case. The infinite-dimensional case is also
studied.
\end{abstract}

\maketitle

\section{Introduction}

In \cite{P}, it was obtained a characterization of linear isometries
between spaces of scalar-valued absolutely continuous functions
defined on compact subsets of the real line. In this paper, we are
interested in obtaining a complete description of maps which
preserve disjointness on spaces of vector-valued absolutely
continuous functions also defined on compact subsets of the real
line. These maps are usually called separating or disjointness
preserving.

\medskip
It is well known that separating linear maps between spaces of
scalar-valued continuous functions defined on compact or locally
compact spaces are automatically continuous and that there exists a
homeomorphism between the underlying spaces (\cite{FH2}, \cite{J},
\cite{JW}). In a more general context, J. J. Font proved that a
separating linear bijection between regular Banach function algebras
which satisfy Ditkin's condition is continuous and their structure
spaces are homeomorphic (\cite{F}).

\medskip
For spaces of vector-valued continuous functions, it is necessary to
require that the inverse map be also separating to obtain a similar
characterization. If a bijective map and its inverse are separating
we call it biseparating. There are several papers which deal with
such maps on spaces of continuous functions and results about
automatic continuity and topological links between the underlying
spaces are obtained (see \cite{AK}, \cite{A1}, \cite{A2}, \cite{AJ},
\cite{GJW} and \cite{HBN}). Nevertheless, we do not know much about
separating maps on spaces of vector-valued continuous functions.
Namely, in spaces of continuous functions vanishing at infinity,
just one result of automatic continuity was given by J. Araujo in
\cite[Theorem 5.4]{A3} for spaces with finite dimension.

\medskip
In this paper, we study bijective and separating linear maps between
spaces of absolutely continuous functions defined on compact subsets
of the real line and taking values in arbitrary Banach spaces. We
obtain a description of such maps which allows us to prove that
their inverses are also separating and to deduce their automatic
continuity in the finite-dimensional case. Besides we show, with an
example, that it is not possible to obtain an analogue result when
we deal with Banach spaces of infinite dimension.  For this reason,
we consider biseparating maps in that case.

\subsection*{Preliminaries and notation}

$X$ and $Y$ will be compact subsets of the real line and $E$ and $F$
will be arbitrary $\mathbb{K}$-Banach spaces, where $\mathbb{K}$
denotes the field of real or complex numbers.\\
If $A$ is a subset of $X$, then $\mathrm{int}(A)$ denotes the
interior of $A$ in $X$, $\mathrm{cl}(A)$ denotes its closure and
$\mathrm{bd}(A)$ its boundary. On the other hand, $\chi_{A}$ denotes
the characteristic function on $A$. Finally, we call \emph{partition}
of $A\subset X$ to any finite family $\{x_{i}\}_{i=0}^{n}$
of points of $A$ which satisfies $x_{0}<x_{1}<\ldots <x_{n}$.\\
Given a function $f:X\rightarrow E$, we define the \emph{cozero set}
of $f$ as $c(f):=\{x\in X: f(x)\neq 0\}$.
Also, for any $x\in X$, we denote by $\delta_{x}$ the functional
\emph{evaluation at the point $x$}, and finally, if $e\in E$, then
$\hat{e}$ will be the constant function from
$X$ to $E$ taking the value $e$.\\
Throughout this paper the word ``homeomorphism" will be synonymous
with ``surjective homeomorphism".

\subsection*{Definitions and previous results}

The space of absolutely continuous functions has been usually
studied in the scalar context, that is, when the functions take real
or complex values (see \cite[Section 18]{HS}). In this part of the
paper, we study it when the functions take values in arbitrary
Banach spaces.

\begin{definition}
A function $f:X\rightarrow E$ is said to be \emph{absolutely
continuous} on $X$ if, given any $\varepsilon>0$, there exists an
$\delta>0$ such that
\begin{displaymath}
\sum_{i=1}^{n}\|f(b_{i})-f(a_{i})\|<\varepsilon
\end{displaymath}
for each finite family of non-overlapping open intervals
$\{(a_{i},b_{i})\}_{i=1}^{n}$ whose extreme points belong to $X$
with
\begin{displaymath}
\sum_{i=1}^{n}(b_{i}-a_{i})<\delta.
\end{displaymath}
$AC(X,E)$ will denote the space of $E$-valued absolutely continuous
functions on $X$. When $E=\mathbb{K}$, we will consider
$AC(X):=AC(X,\mathbb{K})$.
\end{definition}

\begin{definition}
Given $f\in AC(X,E)$, we define the \emph{variation} of $f$ on $X$
as
\begin{displaymath}
V(f;X):=\mathrm{sup}\left\{\sum_{i=1}^n\|f(x_{i})-f(x_{i-1})\|:
\{x_{i}\}_{i=0}^{n}\ \mathrm{is\ a\ partition\ of}\ X,\ n\in
\mathbb{N}\right\}.
\end{displaymath}
\end{definition}

Throughout the paper we will consider $AC(X,E)$
endowed with the norm $\|.\|_{AC}$ defined by:
\begin{center}
$\|f\|_{AC}:=\|f\|_{\infty}+V(f;X)$, for each $f\in AC(X,E)$,
\end{center}
where $\|.\|_{\infty}$ denotes the supremum norm.

\medskip
The next lemmas, whose proofs are straightforward, show some
properties of the space of absolutely continuous functions which
are the key tools to prove some further results. In particular,
Lemma 1.5 proves, for $AC(X)$, the existence of a partition of the
unity (see \cite[Lemma 1]{FH1}).

\begin{lemma}
$(AC(X,E),\|.\|_{AC})$ is a Banach space.
\end{lemma}

\begin{lemma}
Let $f\in AC(X)$ and $g\in AC(X,E)$, then $f\cdot g\in AC(X,E)$.
\end{lemma}

\begin{lemma}
Let $\{V_{i}\}_{i=1}^{n}$ be an open covering of $X$. Then, there
exist $\{f_{i}\}_{i=1}^{n}\subset AC(X)$ with $0\leq f_{i}\leq 1$
and $c(f_{i})\subset V_{i}$, for $i=1,\ldots,n$ such that
$\sum_{i=1}^{n}f_{i}=1$.
\end{lemma}

\section{Separating maps}

\begin{definition}
A map $T:AC(X,E)\rightarrow AC(Y,F)$ is said to be \emph{separating}
if it is linear and $c(Tf)\cap c(Tg)=\emptyset$ whenever $f,g\in
AC(X,E)$ satisfy $c(f)\cap c(g)=\emptyset$.\\
Equivalently, a linear map $T:AC(X,E)\rightarrow AC(Y,F)$ is
\emph{separating} if $\|Tf(y)\|\|Tg(y)\|=0$, for all $y\in Y$,
whenever $f,g\in AC(X,E)$ satisfy $\|f(x)\|\|g(x)\|=0$, for all
$x\in X$.\\
Also, $T$ is said to be \emph{biseparating} if it is bijective and
both $T$ and $T^{-1}$ are separating.
\end{definition}

From now on we will assume that $T:AC(X,E)\rightarrow AC(Y,F)$ is a
separating and bijective map unless otherwise stated.

\begin{definition}
For each $y\in Y$, we define the map $\delta_{y}\circ T:AC(X,E)\rightarrow F$ as
\begin{center}
$(\delta_{y}\circ T)(f):=Tf(y)$, for each $f\in AC(X,E)$.
\end{center}
Therefore, the \emph{support set} associated to
$\delta_{y}\circ T$ is defined by
\begin{center}
$\mathrm{supp}(\delta_{y}\circ T):=\{x\in X:\forall U$ open
neighborhood of $x, \exists f\in AC(X,E)$ with $c(f)\subseteq U$ and
$Tf(y)\neq 0\}.$
\end{center}
\end{definition}

For more details about the next three lemmas see the references
\cite{FH1,J}.

\begin{lemma}
The set $\mathrm{supp}(\delta_{y}\circ T)$ is a singleton for every
$y\in Y$.
\end{lemma}

\begin{definition}
The previous lemma allows us to define a map $h:Y\rightarrow X$ such
that $h(y)$ is the only point that belongs to
$\mathrm{supp}(\delta_{y}\circ T)$, for each $y\in Y$. We call $h$
the \emph{support map} of $T$.
\end{definition}

\begin{lemma}
Given $f\in AC(X,E)$ such that $f\equiv 0$ on an open subset $U$ of
$X$, then $Tf\equiv 0$ on $h^{-1}(U)$.
\end{lemma}

\begin{lemma}
The support map $h$ is continuous and onto.
\end{lemma}

\begin{proposition}
Let $f\in AC(X,E)$ such that $f(h(y))=0$. Then, the following statements
hold:
\begin{enumerate}
\item If $\mathrm{bd}(h^{-1}(h(y)))=\emptyset$, then $Tf\equiv 0$ on
$h^{-1}(h(y))$.
\item If $\mathrm{bd}(h^{-1}(h(y)))\neq \emptyset$,
then $Tf\equiv 0$ on $\mathrm{bd}(h^{-1}(h(y)))$.
\end{enumerate}
\end{proposition}

\begin{proof}
Fix $y_{0}\in Y$ and suppose that $f(h(y_{0}))=0$.\\
(i) If we assume that $\mathrm{bd}(h^{-1}(h(y_{0})))=\emptyset$, we
deduce that $h^{-1}(h(y_{0}))$ is an open and closed set (see
\cite[pp. 24]{E}), and by continuity of $h$, so is $h(y_{0})$. Then,
applying Lemma 2.5, $Tf\equiv 0$ on $h^{-1}(h(y_{0}))$.\\
(ii) In this case, we suppose that
$\mathrm{bd}(h^{-1}(h(y_{0})))\neq \emptyset$. We must see that
$Tf(y)=0$ for every $y\in \mathrm{bd}(h^{-1}(h(y_{0})))$.\\
We consider the following functions in $AC(X,E)$:\\
\hspace*{.5cm}$f_{A}:=f\cdot\chi_{A}$ with $A=(-\infty,h(y_{0}))\cap X$,\\
\hspace*{.5cm}$f_{B}:=f\cdot\chi_{B}$ with $B=(h(y_{0}),\infty)\cap X$,\\
which satisfy that $f=f_{A}+f_{B}$, so $Tf=Tf_{A}+Tf_{B}$.\\
On the other hand, taking into account the definitions of $A$ and
$B$, it is not hard to see that
\begin{center}
$\left[\mathrm{cl}(h^{-1}(A))\backslash h^{-1}(A)\right]\bigcup
\left[\mathrm{cl}(h^{-1}(B))\backslash
h^{-1}(B)\right]=\mathrm{bd}(h^{-1}(h(y_{0})))$,
\end{center}
so we need to prove that $Tf(y)=0$ for each $y\in
\mathrm{cl}(h^{-1}(A))\backslash h^{-1}(A)$ and $y\in
\mathrm{cl}(h^{-1}(B))\backslash h^{-1}(B)$.

\medskip
We next prove that, if $y\in \mathrm{cl}(h^{-1}(A))\backslash
h^{-1}(A)$, then $y\in h^{-1}(h(y_{0}))$ and $Tf(y)=0$. Since $y\in
\mathrm{cl}(h^{-1}(A))$, there exists a sequence $(y_{n})$ in
$h^{-1}(A)$ converging to $y$. By continuity of $h$, we obtain that
$h(y_{n})$ converges to $h(y)$. Besides $\mathrm{cl}(A)\setminus
A=\{h(y_{0})\}$, so $h(y_{n})$ converges to $h(y_{0})$ and then
$h(y)=h(y_{0})$. In order to show that $Tf(y)=0$, we will prove that $Tf_{A}(y)=0$
and $Tf_{B}(y)=0$. By Lemma 2.5, it is obvious that
$Tf_{B}(y_{n})=0$, for each $n\in \mathbb{N}$, and by continuity of
$Tf_{B}$ we deduce that $Tf_{B}(y)=0$.
We now see that $Tf_{A}(y)=0$. Suppose that $Tf_{A}(y)\neq 0$. Let
$(z_{n})$ be a sequence in $h^{-1}(A)$ converging to $y$ and such
that $\|f_{A}(h(z_{n}))\|<\frac{1}{n^{3}}$, for each $n\in
\mathbb{N}$. Taking a subsequence if necessary, we can consider
disjoint open neighborhoods $U_{n}$ of $h(z_{n})$, for each $n\in
\mathbb{N}$, such that $\|f_{A}(x)\|<\frac{1}{n^{3}}$, for all $x\in
U_{n}$ and $V(f_{A};U_{n})<\frac{1}{n^{3}}$. Also, we take compact
neighborhoods $K_{n}$ of $h(z_{n})$ with $K_{n}\subset U_{n}$, for
every $n\in \mathbb{N}$. As each $K_{n}$ is a compact subset of the
real line, we can consider the least compact interval
$[m_{n},M_{n}]$ in $\mathbb{R}$ such that $K_{n}\subset
[m_{n},M_{n}]$, for each $n\in \mathbb{N}$. Then, since each
$K_{n}\subset U_{n}$ and $U_{n}$ is an open set, there exists
$\varepsilon_{n}>0$ satisfying that
$(m_{n}-\varepsilon_{n},m_{n}+\varepsilon_{n})\subset U_{n}$ and
$(M_{n}-\varepsilon_{n},M_{n}+\varepsilon_{n})\subset U_{n}$,  for
every $n\in \mathbb{N}$. Finally, we define $g_{n}\in AC(X)$,
for each $n\in \mathbb{N}$, in the following way:\\
\hspace*{.25cm} $g_{n}\equiv n$ on $[m_{n},M_{n}]\cap X$,\\
\hspace*{.25cm} $g_{n}\equiv 0$ on $\mathrm{X}\backslash (m_{n}-
\frac{\varepsilon_{n}}{2},M_{n}+\frac{\varepsilon_{n}}{2})\cap X$,\\
\hspace*{.25cm} $g_{n}$ is linear on
$(m_{n}-\frac{\varepsilon_{n}}{2},m_{n})\cup
(M_{n},M_{n}+\frac{\varepsilon_{n}}{2})$.\\
Each function $g_{n}$ satisfies that $g_{n}\equiv n$ on $K_{n}$,
$c(g_{n})\subset U_{n}$, $\|g_{n}\|_{\infty}=n$ and $V(g_{n};X)=2n$.
Now, we define the function $g_{0}:=\sum_{n=1}^{\infty}
f_{A}g_{n}$ and we are going to see that $g_{0}$ belongs to $AC(X,E)$.\\
It is enough to see that $\|f_{A}g_{n}\|_{AC}<\frac{4}{n^{2}}$, for
each $n\in \mathbb{N}$. Notice at this point that
$c(f_{A}g_{n})\subset U_{n}$, so we just need to study
$\|f_{A}g_{n}\|_{AC}$ on $U_{n}$, for every $n\in \mathbb{N}$. It is
obvious that $\|f_{A}g_{n}\|_{\infty}<\frac{1}{n^{2}}$ on $U_{n}$,
for each $n\in \mathbb{N}$. On the other hand, if we
consider $\{x_{i}\}_{i=0}^{m}$ any partition of $U_{n}$, we have that\\
\hspace*{.1cm}$\Sigma_{i=1}^{m}\|(f_{A}g_{n})(x_{i})-(f_{A}g_{n})(x_{i-1})\|\leq\\
\hspace*{.1cm}\leq \Sigma_{i=1}^{m}\|(f_{A}
g_{n})(x_{i})-f_{A}(x_{i})g_{n}(x_{i-1})\|+\Sigma_{i=1}^{m}\|f_{A}(x_{i})g_{n}(x_{i-1})-(f_{A}
g_{n})(x_{i-1})\|\\ \hspace*{.1cm}\leq
\|{f_{A}}_{\mid_{U_{n}}}\|_{\infty}\Sigma_{i=1}^{m}|g_{n}(x_{i})-g_{n}(x_{i-1})|+
\|g_{n}\|_{\infty}\Sigma_{i=1}^{m}\|f_{A}(x_{i})-f_{A}(x_{i-1})\|\\
\hspace*{.1cm}\leq
\|{f_{A}}_{\mid_{U_{n}}}\|_{\infty}V(g_{n};U_{n})+\|g_{n}\|_{\infty}V(f_{A};U_{n})<\frac{3}{n^{2}}$,\\
and then $V(f_{A}g_{n};U_{n})\leq\frac{3}{n^{2}}$, for each $n\in
\mathbb{N}$.\\
Let now $V_{n}$ be an open neighborhood of $h(z_{n})$ with
$V_{n}\subset K_{n}$, for every $n\in \mathbb{N}$. It is obvious
that $g_{0}\equiv n\cdot f_{A}$ on $V_{n}$, and, by Lemma 2.5, we
deduce that $Tg_{0}\equiv n\cdot Tf_{A}$ on $h^{-1}(V_{n})$.
Consequently $Tg_{0}(z_{n})=n\cdot Tf_{A}(z_{n})$, for each $n\in
\mathbb{N}$. Taking into account that $Tf_{A}(y)\neq 0$ and the fact
that $Tf_{A}(z_{n})$ converges to $Tf_{A}(y)$, we can conclude that
$\|Tg_{0}(z_{n})\|$ converges to $\infty$. This behavior implies
that $Tg_{0}$ is not continuous, which is absurd.

\medskip
In a similar way, we can see that $y\in h^{-1}(h(y_{0}))$ and
$Tf(y)=0$ when $y\in \mathrm{cl}(h^{-1}(B))\backslash h^{-1}(B)$.
\end{proof}

\section{The finite-dimensional case}

In this section, we study separating bijections between spaces of
absolutely continuous functions that take values in
finite-dimensional normed spaces. We suppose that the spaces $E$ and
$F$ are both $n$-dimensional and $\{e_{1},\ldots,e_{n}\}$ is a basis
of $E$.

\begin{lemma}
Let $f\in AC(X,E)$ such that $f(h(y_{0}))=0$. Then there exists
$y_{1}\in h^{-1}(h(y_{0}))$ such that
$\{T\hat{e_{i}}(y_{1}):i=1,\ldots,n\}$ is a basis of $F$.
\end{lemma}

\begin{proof} By Proposition 2.7, we know that there exists $y_{1}\in
h^{-1}(h(y_{0}))$ such that $Tf(y_{1})=0$. We will prove that
$\{T\hat{e_{i}}(y_{1}):i=1,\ldots,n\}$ is a basis of $F$. As $E$ and
$F$ have the same dimension, it is enough to show
that they are linearly independent.\\
Suppose that $T\hat{e_{1}}(y_{1}),\ldots,T\hat{e_{n}}(y_{1})$ are
linearly dependent. Therefore, we can take $\textmd{f}\in F$ linearly
independent with them and consider the non-vanishing function
$T^{-1}\hat{\textmd{f}}$. Now, as $\{e_{1},\ldots,e_{n}\}$ is a basis of $E$,
there exist $\alpha_{1},\ldots,\alpha_{n}\in \mathbb{K}$, not all of
them equal to zero, such that
$T^{-1}\hat{\textmd{f}}(h(y_{0}))=\sum_{i=1}^{n}\alpha_{i}e_{i}$. For this
reason, if we define the function
$g:=\sum_{i=1}^{n}\alpha_{i}\hat{e_{i}}\in AC(X,E)$, we obtain that
$(T^{-1}\hat{\textmd{f}}-g)(h(y_{0}))=0$, and then $(\hat{\textmd{f}}-Tg)(y_{1})=0$
applying Proposition 2.7 again. This implies that
$\textmd{f}=\sum_{i=1}^{n}\alpha_{i}T\hat{e_{i}}(y_{1})$, which is a
contradiction.
\end{proof}

\begin{theorem}
$h$ is a homeomorphism.
\end{theorem}

\begin{proof} We know that $h$ is a continuous, onto and closed map.
We only need to prove that $h$ is injective. Suppose that there
exist two distinct points $y_{0},y_{1}\in Y$ such that $h(y_{0})=h(y_{1})$ and we will
study the three possible situations.

\medskip
Assume that $y_{0},y_{1}\in \mathrm{bd}(h^{-1}(h(y_{0})))$. If
$\{\textmd{f}_{1},\ldots,\textmd{f}_{n}\}$ is a basis of $F$, since $T$ is an onto
map, we know that there exist $g_{1},\ldots,g_{n}\in AC(X,E)$ such
that $Tg_{i}=\hat{\textmd{f}_{i}}$, for each $i$. We claim that
$g_{1}(h(y_{0})),\ldots,g_{n}(h(y_{0}))$ are linearly independent.
Su\-ppo\-se that it is not true. Therefore, there exist
$\alpha_{1},\ldots,\alpha_{n}\in \mathbb{K}$, not all of them equal
to zero, such that $\sum_{i=1}^{n}\alpha_{i}g_{i}(h(y_{0}))=0$. By
Proposition 2.7, we obtain that
$T\sum_{i=1}^{n}\alpha_{i}g_{i}(y_{0})=0$, and this implies that
$\sum_{i=1}^{n}\alpha_{i}\textmd{f}_{i}=0$, which is not possible.\\
Then, for each $f\in AC(X,E)$, we have that
$f(h(y_{0}))=\sum_{i=1}^{n}\beta_{i}g_{i}(h(y_{0}))$, for
$\beta_{1},\ldots,\beta_{n}\in \mathbb{K}$ not all of them equal to
zero. Applying Proposition 2.7, we obtain that
$Tf(y_{0})=T\sum_{i=1}^{n}\beta_{i}g_{i}(y_{0})=\sum_{i=1}^{n}\beta_{i}\hat{\textmd{f}_{i}}(y_{0})=
\sum_{i=1}^{n}\beta_{i}\textmd{f}_{i}$ and
$Tf(y_{1})=T\sum_{i=1}^{n}\beta_{i}g_{i}(y_{1})=\sum_{i=1}^{n}\beta_{i}\hat{\textmd{f}_{i}}(y_{1})=
\sum_{i=1}^{n}\beta_{i}\textmd{f}_{i}$, that is, $Tf(y_{0})=Tf(y_{1})$ for
each $f\in AC(X,E)$. This behavior implies that $T$ is not onto, in
contradiction with our assumption.

\medskip
Suppose now that $\mathrm{bd}(h^{-1}(h(y_{0})))=\emptyset$. With a
similar reasoning as in the pre\-vious situation we obtain the same
contradiction.

\medskip
Finally, we assume that $y_{0}\in \mathrm{bd}(h^{-1}(h(y_{0})))$ and $y_{1}\in
\mathrm{int}(h^{-1}(h(y_{0})))$. Let $g\in AC(Y,F)$ be a non-zero
function with $c(g)\subset \mathrm{int}(h^{-1}(h(y_{0})))$ and
consider $T^{-1}g$. We are going to prove that there exists an open
subset $V$ of $X$ satisfying that $V\cap \{h(y_{0})\}=\emptyset$ and
$T^{-1}g(x)\neq 0$, for all $x\in V$. If it is not true, $T^{-1}g$
is equal to zero on $X\backslash \{h(y_{0})\}$. Besides, we know
that $h(y_{0})$ is not an isolated point, so we deduce that
$T^{-1}g\equiv 0$ on $X$, which is a contradiction since $g$ is a non-zero function.
Therefore, if we consider $x_{1}\in V$ and a basis
$\{e_{i}:i=1,\ldots,n\}$ of $E$, we have that
$T^{-1}g(x_{1})=\sum_{i=1}^{n}\alpha_{i}\hat{e_{i}}(x_{1})$  for
$\alpha_{1},\ldots,\alpha_{n}\in \mathbb{K}$ not all of them equal
to zero. Applying Proposition 2.7 and Lemma 3.1, we can deduce that
$g(y_{2})=\sum_{i=1}^{n}\alpha_{i}T\hat{e_{i}}(y_{2})\neq 0$ for
some $y_{2}\in h^{-1}(x_{1})$, which is not possible since
$c(g)\subset \mathrm{int}(h^{-1}(h(y_{0})))$.
\end{proof}

\begin{corollary}
Let $f\in AC(X,E)$ such that $f(h(y))=0$. Then $Tf(y)=0$.
\end{corollary}

\begin{proof} It is obvious applying Proposition 2.7 and Theorem 3.2.
\end{proof}

\begin{remark}
For any $y\in Y$ fixed, we define the function
$g:=f-\widehat{f(h(y))}\in AC(X,E)$, for each $f\in AC(X,E)$. It is
obvious that $g(h(y))=0$, and by the previous corollary, we deduce
that $Tg(y)=0$. For this reason, we obtain the next representation
of $T$:
\begin{center}
$Tf(y)=T\widehat{f(h(y))}(y)$, for all $f\in AC(X,E)$ and $y\in Y$.
\end{center}
Therefore, we define the map $J_{y}$, for each $y\in Y$, in the
following way:
$$\begin{array}{ccccc}
J_{y}:&E&\rightarrow&F&\\
&e&\mapsto&J_{y}(e)&:=T\hat{e}(y).
\end{array}$$
\end{remark}

\begin{lemma}
The map $J_{y}$ is linear, bijective and continuous for every $y\in Y$.
\end{lemma}

\begin{proof} Obviously each $J_{y}$ is linear. We next see that $J_{y}$
is a bijective map.\\
Firstly, we will prove that $J_{y}$ is injective. If $e\neq 0$ and
$\{e_{i}:i=1,\ldots,n\}$ is a basis of $E$, then there exist
$\alpha_{1},\ldots,\alpha_{n}\in \mathbb{K}$, not all of them equal
to zero, such that $e=\sum_{i=1}^{n}\alpha_{i}e_{i}$, and this
implies that
$\hat{e}(h(y))=\sum_{i=1}^{n}\alpha_{i}\hat{e_{i}}(h(y))$. By Lemma
3.1 and Corollary 3.3 we deduce that
$T\hat{e}(y)=\sum_{i=1}^{n}\alpha_{i}T\hat{e_{i}}(y)\neq 0$, and by
definition of $J_{y}$ we conclude that $J_{y}(e)\neq 0$.\\
Secondly, we will see that $J_{y}$ is an onto map. Given $\textmd{f}\in F$,
since $T$ is surjective, there exists $g\in AC(X,E)$ such that
$Tg=\hat{\textmd{f}}$, in particular, $Tg(y)=\textmd{f}$. We define $e:=g(h(y))\in E$.
It is obvious that $(\hat{e}-g)(h(y))=0$ and, by Corollary
3.3, we deduce that $T(\hat{e}-g)(y)=0$. This implies that $J_{y}(e)=\textmd{f}$.\\
Finally, it is trivial to see that each $J_{y}$ is continuous since
it is a linear map and $E$ is a finite-dimensional normed space.
\end{proof}

\begin{theorem}
Let $T:AC(X,E)\rightarrow AC(Y,F)$ be a separating and bijective map
with $E$ and $F$ $n$-dimensional normed spaces. Then there exist a
homeomorphism $h:Y\rightarrow X$ and a map $J_{y}:E\rightarrow F$
linear, bijective and continuous for each $y\in Y$, such that
\begin{center}
$Tf(y)=J_{y}(f(h(y)))$
\end{center}
for every $f\in AC(X,E)$ and $y\in Y$.\\
Also, $T$ is continuous and biseparating.
\end{theorem}

\begin{proof} The representation of $T$ follows by Remark 3.4 and
from the definition of $J_{y}$ above. To see that $T$ is a
continuous map we apply the Closed Graph Theorem, so we just need to
prove that $T$ is a closed map (see \cite[Theorem 7.3.2]{L}). Therefore,
it is enough to see that, if we take $(f_{n})$ in $AC(X,E)$
converging to $0$
and $(Tf_{n})$ converges to $g$, then $g\equiv 0$.\\
Firstly, we are going to prove that $\delta_{y}\circ T$ is a
continuous map, for each $y\in Y$. Fix $y\in Y$. It is obvious that
$\delta_{y}\circ T$ is linear, and, by the representation of $T$, we
have that $\|\delta_{y}\circ T(f)\|\leq \|J_{y}\|\|f\|_{\infty}$,
for each $f\in AC(X,E)$. From this inequality, we obtain that
$\delta_{y}\circ T$
is continuous and consequently that $(\delta_{y}\circ T(f_{n}))$ converges to $0$.\\
On the other hand, $\|Tf_{n}(y)-g(y)\|\leq \|Tf_{n}-g\|_{\infty}\leq
\|Tf_{n}-g\|_{AC}$, for each $n\in \mathbb{N}$, and, since we assume
that $(Tf_{n})$ converges to $g$, we deduce that $(Tf_{n}(y))$
converges to $g(y)$ for each $y\in Y$. Combined with the above, we
conclude that $g(y)=0$ for all $y\in Y$ and this completes the
proof that $T$ is continuous.\\
Finally, we prove that $T$ is a biseparating map. It is enough to
see that $T^{-1}:AC(Y,F)\rightarrow AC(X,E)$ is separating. Suppose
that $T^{-1}$ is not separating, then there exist $f,g\in AC(Y,F)$
with $c(f)\cap c(g)=\emptyset$ such that $c(T^{-1}f)\cap
c(T^{-1}g)\neq \emptyset$. For this reason, there exists $x_{0}\in
X$ with $T^{-1}f(x_{0})\neq 0$ and $T^{-1}g(x_{0})\neq 0$. As
$\{e_{1},\ldots,e_{n}\}$ is a basis of $E$, we can take
$\alpha_{1},\ldots,\alpha_{n}\in \mathbb{K}$, not all of them equal
to zero, such that
$T^{-1}f(x_{0})=\sum_{i=1}^{n}\alpha_{i}\hat{e_{i}}(x_{0})$ and
$\beta_{1},\ldots,\beta_{n}\in \mathbb{K}$, not all of them equal to
zero, such that
$T^{-1}g(x_{0})=\sum_{i=1}^{n}\beta_{i}\hat{e_{i}}(x_{0})$. Applying
Lemma 3.1 and Corollary 3.3, we can deduce that
$f(h^{-1}(x_{0}))=\sum_{i=1}^{n}\alpha_{i}T\hat{e_{i}}(h^{-1}(x_{0}))\neq
0$ and
$g(h^{-1}(x_{0}))=\sum_{i=1}^{n}\beta_{i}T\hat{e_{i}}(h^{-1}(x_{0}))\neq
0$, which contradicts the fact that $f$ and $g$ have disjoint cozeros.
\end{proof}

\section{The infinite-dimensional case}

The next example shows that it is not possible to obtain a similar
result as in the previous case when we deal with
infinite-dimensional Banach spaces. For this reason, we study
biseparating maps instead of separating in this case.

\begin{example}
Let $c_{0}$ be the space of all scalar-valued sequences that
converge to zero and let $T:AC([0,1],c_{0})\rightarrow
AC([0,1]\cup[2,3],c_{0})$ be
a bijective, separating and continuous map defined by:\\
\hspace*{.25cm}$Tf(x)=(\lambda_{1},\lambda_{3},\lambda_{5},\ldots)$\\
\hspace*{.25cm}$Tf(2+x)=(\lambda_{2},\lambda_{4},\lambda_{6},\ldots)$\\
when $f(x)=(\lambda_{1},\lambda_{2},\lambda_{3},\ldots)\in c_{0}$,
for each $x\in [0,1]$.\\
It is easy to see that $T^{-1}$ is not a separating map.
\end{example}

\begin{remark}
Similarly to the previous example, it can be constructed a separating bijection
from $AC([0,1],\mathbb{R}^{2})$ to $AC([0,1]\cup[2,3],\mathbb{R})$ which is not
biseparating. This fact allows us to conclude that Theorem 3.6 is not true if we
do not suppose that $E$ and $F$ have the same dimension.
\end{remark}

\begin{remark}
In this final section, $T:AC(X,E)\rightarrow AC(Y,F)$ will be a
bi\-se\-pa\-ra\-ting map and $E$ and $F$ will be Banach spaces of
infinite dimension.\\
Since $T$ is biseparating, we obtain two different continuous support
maps $h$ and $k$ asociated to $T$ and $T^{-1}$, respectively.
\end{remark}

\begin{theorem}
$h$ is a homeomorphism.
\end{theorem}

\begin{proof} It is no difficult to see that $h$ and $k$ are
inverse maps. The proof given in \cite[Theorem 1(8)]{FH1} for group
algebras can be easily adapted to our context.
\end{proof}

\begin{corollary}
Let $f\in AC(X,E)$ such that $f(h(y))=0$. Then $Tf(y)=0$.
\end{corollary}

\begin{proof} It is clear applying Proposition 2.7 and previous theorem.
\end{proof}

\begin{remark}
With the same construction as in Remark 3.4, we obtain the next
representation of $T$:
\begin{center}
$Tf(y)=T\widehat{f(h(y))}(y)$, for all $f\in AC(X,E)$ and $y\in Y$,
\end{center}
and we define the map $J_{y}$, for each $y\in Y$, as in the previous
case.
\end{remark}

\begin{lemma}
$J_{y}$ is linear and bijective for every $y\in Y$.
\end{lemma}

\begin{proof} We obtain that each $J_{y}$ is linear and onto in a
similar way as in the finite-dimensional case. We next prove that
$J_{y}$ is injective. Suppose that it is not true. Thus we consider
$e\in E$ with $e\neq 0$ such that $J_{y}(e)=0$. We have proved that
$k$ is a homeomorphism, so there exists $x\in X$ such that $y=k(x)$,
and then $J_{k(x)}(e)=0$. Since $T\hat{e}(k(x))=0$, applying Corollary
4.5 to the separating map $T^{-1}$, we obtain that
$T^{-1}(T\hat{e})(x)=0$, which implies that $\hat{e}(x)=0$, in contradiction
with $e\neq 0$.
\end{proof}

\begin{theorem}
Let $T:AC(X,E)\rightarrow AC(Y,F)$ be a biseparating map with $E$
and $F$ infinite-dimensional Banach spaces. Then there exist a
homeomorphism $h:Y\rightarrow X$ and a map $J_{y}:E\rightarrow F$
linear and bijective for each $y\in Y$, such that
\begin{center}
$Tf(y)=J_{y}(f(h(y)))$
\end{center}
for every $f\in AC(X,E)$ and $y\in Y$.\\
Also, if $Y$ has no isolated points, then $T$ is continuous.
\end{theorem}

\begin{proof} By Remark 4.6 and the definition of $J_{y}$ we deduce
the representation of $T$. We only need to prove that $T$ is
continuous if $Y$ has no isolated points. We will prove that
$\delta_{y}\circ T$ is continuous for every $y\in Y$, and then,
applying the Closed Graph Theorem in a similar way as in Theorem
3.6, we will deduce that $T$ is a continuous
map.\\
Suppose that there exists $y_{0}\in Y$ such that
$\delta_{y_{0}}\circ T$ is not continuous. Then we consider a
sequence $(e_{n})$ in $E$ such that $\|e_{n}\|\leq \frac{1}{n^{3}}$
and $\|T\hat{e_{n}}(y_{0})\|>1$, for all $n\in \mathbb{N}$. In this
way, we can find a sequence $(y_{n})$ in $Y$, strictly monotone and
converging to $y_{0}$, such that
$\|T\hat{e_{n}}(y_{n})\|>1$, for each $n\in \mathbb{N}$.\\
We now take disjoint open neighborhoods $U_{n}$ of $h(y_{n})$, for
each $n\in \mathbb{N}$, and define $f_{n}\in AC(X)$ such that
$f_{n}(h(y_{n}))=1,\ 0\leq f_{n}\leq 1$ and $c(f_{n})\subset U_{n}$,
for all $n\in \mathbb{N}$. Finally, we consider the function
$f:=\sum_{n=1}^{\infty}f_{n}e_{n}$ that belongs to $AC(X,E)$.\\ It
is obvious that $f(h(y_{0}))=0$ and, by Corollary 4.5,
$Tf(y_{0})=0$. On the other hand, $(f-\hat{e_{n}})(h(y_{n}))=0$ and
then $Tf(y_{n})=T\hat{e_{n}}(y_{n})$, for all $n\in \mathbb{N}$.
This implies that $\|Tf(y_{n})\|>1$, for each $n\in \mathbb{N}$, and
we obtain a contradiction since $Tf$ is continuous.
\end{proof}

\proof[Acknowledgements] The author wishes to thank the suggestions given
by the referee and Professor J. Araujo for his guidance in the development
of this paper.

\end{document}